\documentclass{article}

\usepackage[utf8]{inputenc} 
\usepackage[T1]{fontenc}    
\usepackage{hyperref}       
\usepackage{url}            
\usepackage{booktabs}       
\usepackage{amsfonts}       
\usepackage{nicefrac}       
\usepackage{microtype}      
\usepackage{amsmath}
\usepackage{amsthm}
\usepackage{todonotes}
\usepackage{bm}
\usepackage{xargs}
\usepackage{tikz}
\usepackage{tikz-cd}
\usepackage{dsfont}
\usepackage[margin=1in]{geometry}

\newtheorem{assumption}{Assumption}
\newtheorem{theorem}{Theorem}
\newtheorem{lemma}{Lemma}

\newtheorem{corollary}{Corollary}

\newtheorem{remark}{Remark}

\DeclareMathOperator*{\argmin}{arg\,min}

\newcommand{\prox}[3]{\operatorname{\mathbf{prox}}_{#1,#2} \left ( #3 \right )}

\title{\LARGE \bf Online Stochastic Convex Optimization: \\ Wasserstein Distance Variation}

\author{%
  Iman Shames and Farhad Farokhi
  \\[1em]
  Department of Electrical and Electronic Engineering,\\
  The University of Melbourne, Australia \\[1em]
  e-mails: \texttt{\{iman.shames,farhad.farokhi\}@unimelb.edu.au} \\
}

\begin{document}

\maketitle

\begin{abstract}%
	Distributionally-robust optimization is often studied for a fixed set of distributions rather than time-varying distributions that can drift significantly over time (which is, for instance, the case in finance and sociology due to underlying expansion of economy and evolution of demographics). This motivates understanding conditions on probability distributions, using the Wasserstein distance, that can be used to model time-varying environments. We can then use these conditions in conjunction with online stochastic optimization to adapt the decisions. We considers an online proximal-gradient method to track the minimizers of expectations of smooth convex functions parameterised by a random variable whose probability distributions continuously evolve over time at a rate similar to that of the rate at which the decision maker acts. We revisit the concepts of estimation and tracking error inspired by systems and control literature and provide bounds for them under strong convexity, Lipschitzness of the gradient, and bounds on the probability distribution drift. Further, noting that computing projections for a general feasible sets might not be amenable to online implementation (due to computational constraints), we propose an exact penalty method. Doing so allows us to relax the uniform boundedness of the gradient and establish dynamic regret bounds for tracking and estimation error. We further introduce a constraint-tightening approach and relate the amount of tightening to the probability of satisfying the constraints.
\end{abstract}

\section{Introduction}
In this paper, we are interested in solving problems of the form:
\begin{align*}
\min_{x\in\mathcal{X}\subseteq \mathbb{R}^{n_x}} \mathbb{E}^{w\sim \mathbb{P}_t}\{f(x,w)\},
\end{align*}
where $n_x$ is a positive integer, $\mathcal{X}$ is a closed and convex set, $\mathbb{P}_t$ is a time-varying distribution, and $f$ is convex in its first argument\footnote{The additional requirements on $f$ are presented in Assumption~\ref{assum:assumption}.}. We are particularly interested in solving this sequence of problems without the knowledge of $\mathbb{P}_t$ but with only access to samples $\{w_t^{(i)}\}_{i=1}^m$ from distribution $\mathbb{P}_t$.
These problems can be solved with tools from distributionally-robust optimization and machine learning, see, e.g.,~\cite{esfahani2018data,delage2010distributionally,kuhn2019wasserstein}, so long as $\mathbb{P}_t\in\mathcal{P}$ for all $t\in\mathbb{N}$.
Particularly, if there exists a set $\widehat{\mathcal{P}}$ that can be parameterized using only the samples $\{\{w_t^{(i)}\}_{i=1}^m\}_{t\in\mathbb{N}}$ from the said distributions such that $\mathcal{P}\subseteq \widehat{\mathcal{P}}$, we can replace the original problem with the robust version in
\begin{align*}
\min_{x\in\mathcal{X}\subseteq \mathbb{R}^{n_x}} \sup_{\mathbb{P}\in\widehat{\mathcal{P}}}\mathbb{E}^{w\sim \mathbb{P}}\{f(x,w)\}.
\end{align*}
This way, we can compute the optimal decisions for the worst-case distribution and provide out-of-sample performance guarantees~\cite{esfahani2018data}. However, this is no longer viable for time-varying distributions that  drift significantly over time, i.e., if there does not exists a bounded set $\mathcal{P}$ such that $\mathbb{P}_t\in\mathcal{P}$ for all $t\in\mathbb{N}$ or if such a set $\mathcal{P}$ exists but it is so large that the robust solution becomes extremely conservative. Examples of such problems appear in finance and sociology where the expansion of economy and the evolution of demographics can significantly modify the underlying distributions. This motivates understanding conditions on probability distributions that can be used to model time-varying environments and their effects on online stochastic optimization problems. These are the topics of interest in this paper.

\section{Related Work}
We refer to \cite{doi10113719780898718751} for a detailed treatment of stochastic programming and \cite{nocedal2006numerical} for an overview of convex optimization algorithms.

Our work squarely fits within the online optimization framework \cite{mokhtari2016online,dixit2019online,ajalloeian2019inexact}, where a  decision maker must make a series of decision over $T$ rounds according to some policy. The aim is then to show how suboptimal the decisions are, by adopting a \emph{regret measure}.

To the best of our knowledge, the works of \cite{agrawal2014fast} and \cite{hazan2014beyond} (and especially the latter) are the closest to this work. While those works propose a general framework, we focus on the behaviour of the problem for the specific setup. Particularly, we consider the case where the change in the problem in consecutive time-steps is due to the variation in the distribution of a cost function parameter and use the Wasserstein distance to quantify the change. We study the behaviour of an inexact proximal gradient algorithm where the inexactness is due to a stochastic gradient oracle and derive dynamic regret bounds. Moreover, we do away with the gradient uniform boundedness assumption for a special family of constraint sets (e.g., affine constraints).

Dating back to the pioneering work of~\cite{scarf1958}, distributionally-robust optimization has gained significant attention recently~\cite{ben2009robust,bertsimas2004price}. Distributionally-robust optimization considers stochastic programs where the distribution of the uncertain
parameters is not exactly known and belongs to an ambiguity set.  Several methods for modeling ambiguity sets have been proposed relying on discrete distributions~\cite{postek2016computationally,ben2013robust}, moment constraints~\cite{delage2010distributionally,goh2010distributionally,wiesemann2014distributionally}, Kullback-Leibler divergence~\cite{hu2013kullback,jiang2016data}, Prohorov metric~\cite{erdougan2006ambiguous}, and the Wasserstein distance~\cite{pflug2007ambiguity,wozabal2012framework,esfahani2018data}, among others. Of particular interest in distributionally-robust optimization is the family of data-driven distributionally-robust optimization, where the ambiguity set is parameterized based on samples of the distribution~\cite{chehrazi2010monotone,erdougan2006ambiguous,esfahani2018data,bertsimas2004price}. These studies have proved to be of extreme importance in machine learning~\cite{abadeh2015distributionally,sinha2017certifiable,kuhn2019wasserstein,farokhi2020regularization,mevissen2013data}. Distributionally-robust optimization has also found its way to filtering~\cite{xu2009kalman,levy2012robust, zorzi2016robust}.  These studies focus on fixed ambiguity sets and do not consider drifting distributions. In some cases, we might be able to construct a large enough set to contain the drifting distributions (if they are not moving within an unbounded space) but this approach is often undesirable. This is because, although the ambiguity set must be rich enough to capture the data-generating distributions, it must be small enough to exclude pathological distributions that can make the decision-making process conservative. Such a na\"{i}ve approach in treating drifting distributions can render the solution very conservative. 

In this paper, we opt for Wasserstein distance instead of Kullback-Leibler divergence. This is because Kullback–Leibler ambiguity sets typically fail to represent confidence sets for the unknown distribution because we need to impose absolutely continuity assumptions to get finite Kullback–Leibler divergence. This is clearly not possible when one of the distributions is continuous and the other one is a discrete empirical distribution; see, e.g.,~\cite{esfahani2018data} for more discussion.

\section{Wasserstein Distance}\label{sec:Wasserstein}
We use the notation $\mathcal{M}(\Xi)$ to denote the set of all probability distributions $\mathbb{Q}$ on $\Xi\subseteq\mathbb{R}^n$ such that $\mathbb{E}^{\xi\sim\mathbb{Q}}\{\|\xi\|\}<\infty$ for some norm $\|\cdot\|$ on $\mathbb{R}^m$. When evident from the context, we replace $\mathbb{E}^{\xi\sim\mathbb{Q}}\{\cdot\}$ with $\mathbb{E}^{\mathbb{Q}}\{\cdot\}$. For all $\mathbb{P}_1,\mathbb{P}_2\in\mathcal{M}(\Xi)$, the Wasserstein distance is
\begin{align*}
\mathfrak{W}(\mathbb{P}_1,\mathbb{P}_2):=\inf \Bigg\{\int_{\Xi^2} \|\xi_1-\xi_2\|\Pi(\mathrm{d}\xi_1,\mathrm{d}\xi_2):
& \Pi \mbox{ is a joint disribution on } \xi_1 \mbox{ and }\xi_2\\[-1em]
& \mbox{with marginals }\mathbb{P}_1\mbox{ and }\mathbb{P}_2\mbox{, respectively}\Bigg\}.
\end{align*}
Let us define $\mathbb{L}$ as the set of all Lipschitz functions with Lipschitz constant upper bounded by one, i.e., all functions $f$ such that $|f(\xi_1)-f(\xi_2)|\leq \|\xi_1-\xi_2\|$ for all $\xi_1,\xi_2\in\Xi$.
The Wasserstein distance can be alternatively computed using
\begin{align}
\mathfrak{W}(\mathbb{P}_1,\mathbb{P}_2):=\sup_{f\in\mathbb{L}} \Bigg\{&\hspace{-.03in}\int_{\Xi} f(\xi)\mathbb{P}_1(\mathrm{d}\xi)-\int_{\Xi} f(\xi)\mathbb{P}_2(\mathrm{d}\xi)\Bigg\}. \label{eqn:dual}
\end{align}
We refer interested readers to~\cite{kantorovich1958space} for more information.
\section{Online Stochastic Optimization}
We want to sequentially solve the following problems:
\begin{align}\label{eq:main_problem}
\min_{x\in\mathcal{X}\subseteq \mathbb{R}^{n_x}} \underbrace{\mathbb{E}^{w\sim \mathbb{P}_t}\{f(x,w)\}}_{:=F_t(x)}.
\end{align}
The time varying aspect of the optimization problem is the distribution $\mathbb{P}_t$. We make the following standing assumptions.

\begin{assumption} \label{assum:assumption} The following statements hold:
	\begin{itemize}
		\item[(a)] $f(x,w)$ is twice continuously differentiable with respect to $x$;
		\item[(b)] $f(x,w)$, $\partial f(x,w)/\partial x_i$, $\partial^2 f(x,w)/\partial x_i\partial x_j$ are integrable with respect to $w$ for all $i,j$;
		\item[(c)] $\mathbb{E}^{w\sim \mathbb{P}_t}\{|f(x,w)|\}$, $\mathbb{E}^{w\sim \mathbb{P}_t}\{|\partial f(x,w)/\partial x_i|\}$, and $\mathbb{E}^{w\sim \mathbb{P}_t}\{|\partial^2 f(x,w)/\partial x_i\partial x_j|\}$
		are bounded for all $i$, $j$, and $t$;
		\item[(d)] $\mathfrak{W}(\mathbb{P}_{t+1},\mathbb{P}_t)\leq \rho$ for all $t\in\mathbb{N}$;
		\item[(e)] $\|\nabla_x f(x,w)-\nabla_x f(x',w)\|\leq L_x\|x-x'\|$ for all $w$;
		\item[(f)] $|\partial f(x,w)/\partial x_i-\partial  f(x,w')/\partial x_i|\leq L_w\|w-w'\|$ for all $x$ and all $i$;
		\item[(g)] $\nabla_x^2 f(x,w)\succeq \sigma I$.
	\end{itemize}
\end{assumption}

These assumptions are common in online stochastic optimization~\cite{doi10113719780898718751}. Most of them are required to keep the changes in the optimizers of the time-varying optimization problem small enough so that we can track it with bounded error. We have added additional assumptions on probability distribution drift using the Wasserstein distance to use the tools from distributionally-robust optimization in online stochastic optimization. Again, it is worth mentioning that we use  Wasserstein distance instead of Kullback-Leibler divergence because, for getting finite Kullback–Leibler divergence, we need to impose absolutely continuity assumptions. This does not hold when one of the distributions is continuous and the other one is a discrete empirical distribution. Note that, in this paper, we only have access to random samples and do not know the underlying distributions.

%
%
%
%
%
%
\begin{lemma}\label{lemma:Lipschitzness} There exists a positive scalar $L_x$ such that  $\|\nabla_x F_t(x)-\nabla_x F_t(x')\|\leq L_x\|x-x'\|$.
\end{lemma}

\begin{proof}
	See Appendix~\ref{proof:lemma:Lipschitzness}.
\end{proof}

\begin{lemma} \label{lem:opt_drift} Let $x^*_t := \argmin_{x\in\mathcal{X}}\;  F_t(x)$, then  $\|x^*_{t+1}-x^*_t\|\leq v$ for all $t\in\mathbb{N}$ where $v:=\rho n_x L_w/\sigma$.
\end{lemma}

\begin{proof}
	See Appendix~\ref{proof:lem:opt_drift}.
\end{proof}


At each time $t$, we get $m$ realizations of random variable $w$ distributed according to $\mathbb{P}_t$ denoted by $\{w_t^{(i)}\}_{i=1}^m$. Based on all the realizations received in the past $q$ time steps, we can construct the empirical density function
\begin{align*}
\widehat{\mathbb{P}}_t=\frac{1}{mq}\sum_{k=t-q+1}^t \sum_{i=1}^m \delta_{w_k^{(i)}},
\end{align*}
where $\delta_\xi$ is the Dirac distribution function with its mass concentrated at $\xi$, i.e., $\int \delta_\xi(w)\mathrm{d}w=1$ and $\delta_\xi(w)=0$ for all $w\neq \xi$. The approximation of the gradient is given by
\begin{align}\label{eq:grad_approx}
\mathbb{E}^{w\sim \widehat{\mathbb{P}}_{t}}\{\nabla_x f(x,w)\}=\frac{1}{mq}\sum_{k=t-q+1}^t \sum_{i=1}^m
\nabla_x f(x,w_k^{(i)}) := \eta_t(x).
\end{align}
In what follows, we show how good of an approximation this is. To do so, we require the following standing assumption.

\begin{assumption} \label{assum:subgaussian}
	There exists constant $a>1$ such that $\mathbb{E}^{w\sim \mathbb{P}_k}\{\exp(\|w\|^a)\}<\infty$.
\end{assumption}

This assumption implies that distributions $\mathbb{P}_t$ are light-tailed. All distributions with a compact support set satisfy this assumption. This assumption is often implicit in the machine learning literature because the empirical average of the loss based on the samples may not even converge to the expected loss in general for heavy-tailed distributions~\cite{brownlees2015empirical,catoni2012challenging}.

\begin{lemma}\label{lem:error_bound_prob}
	The following inequality holds $$\mathbb{P}\{\|\nabla F_t(x_t) - \eta_t(x_t)\|\leq L_w(\zeta(\gamma)+(q-1)\rho/2)\}\geq 1-qn_x\gamma,$$
	where
	\begin{align*}
	\zeta(\gamma):=\begin{cases}
	\displaystyle\left(\frac{\log(c_1/\gamma)}{c_2m}\right)^{1/\max\{n_w,2\}}, & \displaystyle m\geq \frac{\log(c_1/\gamma)}{c_2}, \\
	\displaystyle\left(\frac{\log(c_1/\gamma)}{c_2m}\right)^{1/a}, & \displaystyle m< \frac{\log(c_1/\gamma)}{c_2},
	\end{cases}
	\end{align*}
	for all $m\geq 1$, $n_w\neq 2$, and $\gamma>0$.
\end{lemma}

\begin{proof}
	By convexity of the Wasserstein distance~\cite[Lemma~2.1]{pflug2014multistage} , we get
	\begin{align*}
	\mathfrak{W}(\widehat{\mathbb{P}}_t,\mathbb{P}_t)
	\leq \frac{1}{q}\sum_{k=t-q+1}^t \mathfrak{W}\left(\frac{1}{m} \sum_{i=1}^m \delta_{w_k^{(i)}},\mathbb{P}_t\right).
	\end{align*}
	On the other hand, the triangle inequality for the Wasserstein distance~\cite[p.\,170]{prugel2020probability} implies that
	\begin{align*}
	\mathfrak{W}\left(\frac{1}{m} \sum_{i=1}^m \delta_{w_k^{(i)}},\mathbb{P}_t\right)
	&\leq
	\mathfrak{W}\left(\frac{1}{m} \sum_{i=1}^m \delta_{w_k^{(i)}},\mathbb{P}_k\right)+\mathfrak{W}(\mathbb{P}_k ,\mathbb{P}_t)\\
	&\leq
	\mathfrak{W}\left(\frac{1}{m} \sum_{i=1}^m \delta_{w_k^{(i)}},\mathbb{P}_k\right)+\mathfrak{W}(\mathbb{P}_k ,\mathbb{P}_{k+1})+\cdots+\mathfrak{W}(\mathbb{P}_{t-1} ,\mathbb{P}_{t}),
	\end{align*}
	and, as a result,
	\begin{align*}
	\mathfrak{W}\left(\frac{1}{m} \sum_{i=1}^m \delta_{w_k^{(i)}},\mathbb{P}_t\right)
	&\leq
	\mathfrak{W}\left(\frac{1}{m} \sum_{i=1}^m \delta_{w_k^{(i)}},\mathbb{P}_k\right)+(t-k)\rho.
	\end{align*}
	Following~\cite{fournier2015rate} and~\cite{esfahani2018data}, we know that
	\begin{align*}
	\mathbb{P}\left\{\mathfrak{W}\left(\frac{1}{m} \sum_{i=1}^m \delta_{w_k^{(i)}},\mathbb{P}_k\right)
	\leq \zeta(\gamma)\right\}
	\geq 1-\gamma.
	\end{align*}
	Therefore, with probability $1-q\gamma$, we get
	\begin{align*}
	\mathbb{P}\left\{\mathfrak{W}(\widehat{\mathbb{P}}_t,\mathbb{P}_t)
	\leq \frac{1}{q}\sum_{k=t-q+1}^t (\zeta(\gamma)+(t-k)\rho)\right\}\geq 1-q\gamma,
	\end{align*}
	and, as a result,
	\begin{align*}
	\mathbb{P}\left\{\mathfrak{W}(\widehat{\mathbb{P}}_t,\mathbb{P}_t)
	\leq \zeta(\gamma)+(q-1)\rho/2\right\}\geq 1-q\gamma,
	\end{align*}
	Following~\eqref{eqn:dual} with Assumption~\ref{assum:assumption}~(f), we get
	\begin{align*}
	|\mathbb{E}^{w\sim \widehat{\mathbb{P}}_{t}}\{\partial f(x,w)/\partial x_i\}- \mathbb{E}^{w\sim \mathbb{P}_t}\{\partial  f(x,w)/\partial x_i\}|
	\leq \mathfrak{W}(\mathbb{P}_t, \widehat{\mathbb{P}}_{t}) L_w,
	\end{align*}
	which implies that
	\begin{align*}
	\mathbb{P}\{|\mathbb{E}^{w\sim \widehat{\mathbb{P}}_{t}}\{\partial f(x,w)/\partial x_i\}- \mathbb{E}^{w\sim \mathbb{P}_t}\{\partial  f(x,w)/\partial x_i\}|\leq L_w(\zeta(\gamma)+(q-1)\rho/2)\}\geq 1-q\gamma.
	\end{align*}
	Therefore,
	\begin{align*}
	\mathbb{P}\{\|\mathbb{E}^{w\sim \widehat{\mathbb{P}}_{t}}\{\nabla_x f(x,w)\}-& \mathbb{E}^{w\sim \mathbb{P}_t}\{\nabla_x f(x,w)\}\|\leq L_w(\zeta(\gamma)+(q-1)\rho/2)\}\geq 1-qn_x\gamma.
	\end{align*}
	This concludes the proof.
\end{proof}

It is more useful to provide a bound on the expectation of the norm of the gradient error. This is done in the following theorem.

\begin{theorem} \label{cor:error_bound}
	The norm of the difference between the gradient $\nabla F_t(x_t)$ and $\eta_t(x_t)$, its approximation given by \eqref{eq:grad_approx}, is uniformly bounded in expectation. Specifically, $\mathbb{E}\{\|\nabla F_t(x_t) - \eta_t(x_t)\|\}\leq \Delta,$ where
	$$\Delta := \dfrac{L_w(q-1)\rho}{2} +q n_xc_1L_w\left(\dfrac{\Gamma\left(\frac{1}{\max\{2,n_w\}}\right)}{\max\{2,n_w\}(c_2m)^{1/\max\{2,n_w\}}}+\dfrac{\Gamma\left(\frac{1}{a}\right)}{a(c_2m)^{1/a}}\right),$$
	and $\Gamma(z)=\int_0^\infty \tau^{z-1} \exp(-\tau) \mathrm{d}\tau$ is the gamma function.
\end{theorem}

\begin{proof} 
	Note that
	\begin{align*}
	\mathbb{E}\{\|\mathbb{E}^{w\sim \widehat{\mathbb{P}}_{t}}&\{\nabla_x f(x,w)\}- \mathbb{E}^{w\sim \mathbb{P}_t}\{\nabla_x f(x,w)\}\|\}\\
	=&\int_0^\infty \mathbb{P}\{\|\mathbb{E}^{w\sim \widehat{\mathbb{P}}_{t}}\{\nabla_x f(x,w)\}- \mathbb{E}^{w\sim \mathbb{P}_t}\{\nabla_x f(x,w)\}\|\geq t\} \mathrm{d}t\\
	=&\int_0^{L_w(q-1)\rho/2} \mathbb{P}\{\|\mathbb{E}^{w\sim \widehat{\mathbb{P}}_{t}}\{\nabla_x f(x,w)\}- \mathbb{E}^{w\sim \mathbb{P}_t}\{\nabla_x f(x,w)\}\|\geq t\} \mathrm{d}t\\
	&+\int_{L_w(q-1)\rho/2}^\infty \mathbb{P}\{\|\mathbb{E}^{w\sim \widehat{\mathbb{P}}_{t}}\{\nabla_x f(x,w)\}- \mathbb{E}^{w\sim \mathbb{P}_t}\{\nabla_x f(x,w)\}\|\geq t\} \mathrm{d}t\\
	\leq &\int_0^{L_w(q-1)\rho/2} \hspace{-.1in}\mathrm{d}t
	+\int_{0}^\infty \mathbb{P}\{\|\mathbb{E}^{w\sim \widehat{\mathbb{P}}_{t}}\{\nabla_x f(x,w)\}- \mathbb{E}^{w\sim \mathbb{P}_t}\{\nabla_x f(x,w)\}\|\geq t+L_w(q-1)\rho/2\} \mathrm{d}t.
	\end{align*}
	Following Lemma~\ref{lem:error_bound_prob}, we have
	\begin{align*}
	\mathbb{P}\{\|\mathbb{E}^{w\sim \widehat{\mathbb{P}}_{t}}\{\nabla_x f(x,w)\}- \mathbb{E}^{w\sim \mathbb{P}_t}\{\nabla_x f(x,w)\}\|\geq t+L_w(q-1)\rho/2\}\leq qn_x\zeta^{-1} (t/L_w),
	\end{align*}
	where
	\begin{align*}
	\zeta^{-1}(\varepsilon)
	=
	\begin{cases}
	c_1\exp(-c_2m\varepsilon^{\max\{2,n_w\}}), & \varepsilon\leq 1,\\
	c_1\exp(-c_2m\varepsilon^a), & \varepsilon>1.\\
	\end{cases}
	\end{align*}
	Therefore,
	\begin{align*}
	\int_{0}^\infty \mathbb{P}\{\|&\mathbb{E}^{w\sim \widehat{\mathbb{P}}_{t}}\{\nabla_x f(x,w)\}- \mathbb{E}^{w\sim \mathbb{P}_t}\{\nabla_x f(x,w)\}\|\geq t+L_w(q-1)\rho/2\} \mathrm{d}t\\
	\leq &\int_{0}^{L_w} \mathbb{P}\{\|\mathbb{E}^{w\sim \widehat{\mathbb{P}}_{t}}\{\nabla_x f(x,w)\}- \mathbb{E}^{w\sim \mathbb{P}_t}\{\nabla_x f(x,w)\}\|\geq t+L_w(q-1)\rho/2\} \mathrm{d}t\\
	&+\int_{L_w}^\infty \mathbb{P}\{\|\mathbb{E}^{w\sim \widehat{\mathbb{P}}_{t}}\{\nabla_x f(x,w)\}- \mathbb{E}^{w\sim \mathbb{P}_t}\{\nabla_x f(x,w)\}\|\geq t+L_w(q-1)\rho/2\} \mathrm{d}t\\
	\leq &\int_{0}^{L_w} q n_xc_1\exp(-c_2m(t/L_w)^{\max\{2,n_w\}}) \mathrm{d}t+\int_{L_w}^\infty q n_x c_1\exp(-c_2m(t/L_w)^a) \mathrm{d}t\\
	\leq &\int_{0}^{\infty} q n_xc_1\exp(-c_2m(t/L_w)^{\max\{2,n_w\}})\mathrm{d}t+\int_{0}^\infty q n_x c_1\exp(-c_2m(t/L_w)^a) \mathrm{d}t\\
	= &q n_xc_1L_w\left(\Gamma\left(\frac{1}{\max\{2,n_w\}}\right)\frac{1}{\max\{2,n_w\}(c_2m)^{1/\max\{2,n_w\}}}+\Gamma\left(\frac{1}{a}\right)\frac{1}{a(c_2m)^{1/a}}\right).
	\end{align*}
	This concludes the proof.
\end{proof}

Both Lemma~\ref{lem:error_bound_prob} and Theorem~\ref{cor:error_bound} show that the the gradient approximation $\eta_t(x_t)$ becomes closer to the true gradient $\nabla F_t(x_t)$ when $m$ increases. However, the effect of $q$ is in the opposite. This is because we are considering drifting distributions and, in the worst case, $\mathbb{P}_t$ can be very far from $\mathbb{P}_{t-q+1}$.
In fact, in the worst case, we can have $\mathfrak{W}(\mathbb{P}_t,\mathbb{P}_{t-q+1})=q\rho$. Therefore, using old samples can potentially backfires if $q$ is large. Theoretically, the best choice is to set $q=1$. However, in practice, it might be suitable to select larger $q$ if the distributions are not following the worst-case drift.

In what follows, we briefly review the properties of proximal gradient methods applied in to our problem of interest using the online Optimization framework. Consider the following time-varying problem
\begin{align}
\min_{x\in\mathbb{R}^{n_x}} \quad F_t(x) + h(x)
\end{align}
where $F_t$ is defined as before and $h$ is a closed and proper function; see~\cite{parikh2014proximal} for definitions.
Define the proximal operator of a scaled function $\alpha h$ for a closed and proper function $h$ where $\alpha >0$ as
\begin{align}
\prox{h}{\alpha}{u} = \argmin_{x\in\mathbb{R}^{n_x}} \left (h(x) + \dfrac{1}{2\alpha} \|x-u\|^2 \right ).
\end{align}
The proximal gradient iterations then become
\begin{align} \label{eq:grad_proj_prox}
x_{t+1} = \prox{h}{\alpha_t}{x - \alpha_t \eta_t(x)}.
\end{align}
It is known that if $F_t = F$ for all $t$, $\eta(x) = \nabla F(x)$, and $h(x) = I_{\mathcal{X}}(x)$ where $I_\mathcal{X}$ is the indicator function of $\mathcal{X}$, i.e. it returns 0 if $x\in\mathcal{X}$ and returns $+\infty$ otherwise, then the iterations converge to the solution of the time-invariant problem $\min_{x\in\mathcal{X}} F(x)$.
However, in what follows, we do not limit ourselves to this scenario and the stated results hold for the case where \eqref{eq:grad_proj_prox} is applied to an online optimization problem for any choice of $h$.
Let $\epsilon_t$ be the error between $\eta_t(x_t)$ and $\nabla F_t(x_t)$, i.e.,
\begin{align}
\epsilon_t = \nabla F_t(x_t) - \eta_t(x_t).
\end{align}



\begin{lemma}\label{lem:time_wise_est_bound}
	For $\alpha_t\in(0,2/L_x)$, $\|x_{t+1}-x^*_t\| \leq r_t \|x_t-x^*_t\| + \alpha_t \|\epsilon_t\|$ where $r_t = \max(|1-\alpha_t L_x|, |1-\alpha_t \sigma|)$.
\end{lemma}

\begin{proof}
	See Appendix~\ref{proof:lem:time_wise_est_bound}.
\end{proof}

Before, continuing further, we comment on a nuance that is often neglected in the application of online optimization methods. Depending on the context,  one may want to use either $x_t$ or $x_{t+1}$ as a surrogate for $x^*_t$ at time $t$. If the aim is to take an action at time $t$ without being able to compute \eqref{eq:grad_proj_prox}, one needs to use $x_t$. However, if one can delay taking an action by one time-step, it is prudent to use the value $x_{t+1}$ obtained from \eqref{eq:grad_proj_prox}.
Borrowing some terminology from control systems, we define the \emph{tracking error} as $e_t:=\|x_t-x^*_t\|$, and the \emph{estimation error} as $\bar{e}_t:=\|x_{t+1}-x^*_t\|$\footnote{This noncausal definition might make some readers uncomfortable. We urge that they accept this for the sake of clarity of presentation. Otherwise, one may define $\bar{e}_t:=\|x_{t}-x^*_{t-1}\|$. In this case however, if the goal is to discuss the distance to $x^*_t$, we would need to use $e_t$ and $\bar{e}_{t+1}$ which result in a less clear presentation.}.
We have the following observation
\begin{align}
\|x_{t+1}-x^*_{t+1}\| & \leq \|x_{t+1}-x^*_t\| + \|x^*_{t+1} - x^*_t\| \notag\\
& \leq r_t \|x_t-x^*_t\| + \alpha_t \|\epsilon_t\| + \|x^*_{t+1} - x^*_t\|. \label{eq:time_wise_bound}
\end{align}
\begin{lemma}\label{lem:track}
	Let $\alpha$ be a constant scalar in $(0,2/L_x)$. Then
	\begin{align}
	\underset{t\rightarrow \infty}{\lim\sup} \quad \mathbb{E}[ e_t ] &\leq \dfrac {v + \alpha\Delta}{1-r}, \label{eq:asymp_track_err}\\
	\underset{t\rightarrow \infty}{\lim\sup} \quad \mathbb{E}[ \bar{e}_t ] &\leq \dfrac {rv + \alpha\Delta}{1-r}, \label{eq:asymp_est_err}
	\end{align}
	where $v$ and $\Delta$ are given in Lemma~\ref{lem:opt_drift} and Theorem~\ref{cor:error_bound}, respectively.
\end{lemma}
\begin{proof}
	Recursively applying \eqref{eq:time_wise_bound}, taking the expectation from both sides, and using the bounds $v$ and $\Delta$ from Lemma~\ref{lem:opt_drift} and Theorem~\ref{cor:error_bound}, result in
	$
	\mathbb{E}[ \|x_t-x^*_t\|]   \leq r^{t-1} \|x_1-x^*_1\| + (\alpha \Delta + v) \sum_{k=1}^{t-1} r^{k-1}.
	$
	The first inequality is obtained by computing the right hand-side of this equation as $t\rightarrow \infty$. The second inequality is a consequence of Lemma~\ref{lem:time_wise_est_bound} and \eqref{eq:asymp_track_err}.
\end{proof}

Lemma~\ref{lem:track} shows that we can closely follow the optimizers of the time-varying optimization problem in~\eqref{eq:main_problem}. We only require few samples in each iteration to do this. The effect of the number of the samples is hidden in the definition of $\Delta$ in Theorem~\ref{cor:error_bound}. The estimation and tracking error bounds in Lemma~\ref{lem:track} are only presented in the limit, i.e., as $t$ tends to infinity.

To provide  meaningful bounds on regret, we need to make the following assumption. Note that this assumption is a stronger version of Assumption~\ref{assum:assumption} (c). We \textit{will remove this assumption later} in the paper.

\begin{assumption}\label{assum:bounded_grad}
	For all $t$ and $x\in\mathcal{X}$, there exists a positive scalar $G$ such that $\|\nabla F_t(x)\| \leq G$.
\end{assumption}

\begin{theorem}\label{thm:regret}
	Let ${\textup{Reg}}_T^{\textup{Tracking}}:= \sum_{t=1}^T F_t(x_t) - F_t(x^*_t)$ and  ${\textup{Reg}}_T^{\textup{Estimation}}:= \sum_{t=1}^T F_t(x_{t+1}) - F_t(x^*_t)$ denote the tracking and estimation dynamic regrets, respectively. Then, under Assumptions \ref{assum:assumption} -- \ref{assum:bounded_grad},
	\begin{align}\label{eq:bound_track_thm}
	{\textup{Reg}}_T^{\textup{Tracking}} \leq \dfrac{G}{1-r}  \left [ (e_1 -r e_T - \alpha\|\epsilon_T\| ) + \alpha E_T + V_T \right ],
	\end{align}
	and
	\begin{align}\label{eq:bound_est_thm}
	{\textup{Reg}}_T^{\textup{Estimation}} \leq \dfrac{G}{1-r}  \left [ (\bar{e}_1 -r \bar{e}_T - \alpha \|\epsilon_1\| ) + \alpha E_T +r V_T \right ].
	\end{align}
	where $V_T=\sum_{t=1}^{T-1} \|x^*_{t+1} - x^*_t\|$ and $E_T=\sum_{t=1}^T \|\epsilon_k(x_t)\|$.
\end{theorem}

\begin{proof}
	See Appendix~\ref{proof:thm:regret}.
\end{proof}
The regret bounds presented in Theorem~\ref{thm:regret} depend on $V_T$ and $E_T$. Without extra (anti-causal) information or enforcing extra assumptions on the problem structure, it is impossible improve the dependence of the bound on these values. Obviously, if $E_T$ and $V_T$ are  sub-linear in $T$, the presented regret bounds will be sublinear in $T$ as well.

The right-hand side of \eqref{eq:bound_track_thm} and \eqref{eq:bound_est_thm} are upper bounded  by $$U_T^{\textup{Tracking}}(\alpha) := \dfrac{G}{1-r}  \left [ (e_1 - \alpha\|\epsilon_T\| ) + \alpha E_T + V_T \right ]$$ and
$$U_T^{\textup{Estimation}}(\alpha) :=\dfrac{G}{1-r}  \left [ (\bar{e}_1 - \alpha \|\epsilon_1\| ) + \alpha E_T +r V_T \right ],$$ respectively. Using Lemma~\ref{lem:time_wise_est_bound} and after some algebraic manipulations it can be observed that the choice of $\alpha = 2/(L_x +\sigma)$ minimizes these bounds.
This is the same choice of $\alpha$ that minimizes $r$ and consequently the right-hand side of the inequalities given in Lemma~\ref{lem:track}.

Note that implementing \eqref{eq:grad_proj_prox} where $h(x)=I_\mathcal{X}(x)$ requires computing a projection onto $\mathcal{X}$. Computing this projection for a general set might not be amenable to an online implementation. Thus, to propose a method suitable for online Optimization for a broader family of constraints we make the following assumption.
We have the following assumption.

\begin{assumption}\label{assum:constraints}
	Let $\mathcal{X} = \left \{x|c_i(x) \geq 0, i\in\mathcal{I} \right \}$ where $\mathcal{I} = \{1,\dots,n_c\}$ is the set of indices for the constraints with $n_c$ being a positive integer, $c_i(x)$, $i\in\mathcal{I}$, are continuously differentiable and Lipschitz with Lipschitz constant $L_c$.
\end{assumption}

\begin{remark}
	Affine constraints of the form $c_i(x):= a_i^\top x - b_i$, $\forall i\in\mathcal{I}$,   satisfy Assumption~\ref{assum:constraints} for $L_c=\max_{i\in\mathcal{I}} \|a_i\|$.
\end{remark}

In what follows instead of \eqref{eq:main_problem}, we consider the following time-varying problem
\begin{align}
\label{eq:l1pen_problem}
\min_{x\in \mathbb{R}^{n_x}} F_t(x) + h(x),
\end{align}
where $h(x)= \lambda \sum_{i=1}^m \max (- c_i(x), 0 )$. It is straightforward to see
\begin{align}
|h(x) - h(x')| \leq \lambda n_c L_c \|x-x'\|,\quad \forall x,x'\in\mathbb{R}^{n_x}.
\end{align}
From~\cite{nocedal2006numerical,pietrzykowski1969exact,zangwill1967non}, we know that there exists a bounded $\bar{\lambda}$ such that for all $\lambda \geq \bar{\lambda}$, $\bar{x}_t = x^*_t$ where $\bar{x}_t = \arg  \min_{x\in \mathbb{R}^{n_x}} F_t(x) + h(x)$.
We state the following result.
\begin{lemma} \label{lem:dont_know_what_to_call_it}
	Let $\Phi_t(x) = F_t(x) + h(x)$ and $x_t$ be the sequence generated by \eqref{eq:grad_approx} for \eqref{eq:l1pen_problem} and $\varphi_t$ be an arbitrary element of the subdifferential of $\Phi_t$ at $x_t$ denoted by $\partial \Phi_t(x_t)$. Then,
	\begin{align}
	\mathbb{E} [  \|\varphi_t\|] \leq \overline{G},
	\end{align}
	where \begin{align}\label{eq:subgrad_bound}
	\overline{G} = \dfrac{1+r}{\alpha} \left [ r e_1 + \alpha \Delta + \dfrac{\alpha \Delta + v}{1-r} \right ] + 2\Delta +2\lambda n_c L_c.
	\end{align}
\end{lemma}

\begin{proof}
	See Appendix~\ref{proof:lem:dont_know_what_to_call_it}.
\end{proof}

The following theorem does not require the strong assumption in Assumption~\ref{assum:bounded_grad}. This makes the regret bounds more versatile.

\begin{theorem}\label{thm:regret_l1pen}
	Let $\overline{\textup{Reg}}_T^{\textup{Tracking}}:= \sum_{t=1}^T \mathbb{E} [\Phi_t(x_t)] - \Phi_t(x^*_t)$ and  $\overline{\textup{Reg}}_T^{\textup{Estimation}}:= \sum_{t=1}^T \mathbb{E} [\Phi_t(x_{t+1})] - \Phi_t(x^*_t)$ denote the tracking and estimation dynamic regrets, respectively.
	Then, under Assumptions \ref{assum:assumption} -- \ref{assum:subgaussian} and~\ref{assum:constraints},
	\begin{align}
	\overline{\textup{Reg}}_T^{\textup{Tracking}} \leq \dfrac{\overline{G}}{1-r}  \left [ e_1  -\alpha  \mathbb{E} [\|\epsilon_T\|] + \alpha \overline{E}_T +V_T\|\right ],
	\end{align}
	and
	\begin{align}
	\overline{\textup{Reg}}_T^{\textup{Estimation}} \leq \dfrac{\overline{G}}{1-r}  \left [ r e_1  + \alpha  \overline{E}_T +r V_T\right ],
	\end{align}
	where $\overline{G}$ is given in \eqref{eq:subgrad_bound}, $V_T=\sum_{t=1}^{T-1} \|x^*_{t+1} - x^*_t\|$, and $ \overline{E}_T=\sum_{t=1}^T \mathbb{E} [\|\epsilon_t(x_t)\|]$.
\end{theorem}

\begin{proof}
	The proof is nearly identical to that of Theorem~\ref{thm:regret} with the exception that in the first line, one has to use a subgradient instead of the gradient. The rest of the result follows from taking the expectation from both sides of the inequality and the Cauchy-Schwarz inequality for expectation.\end{proof}

Next, we introduce a problem modification in the form of a \emph{constraint tightening} that results in an asymptotically feasible solution. Note that encoding the constraints as a penalty function as described in \eqref{eq:l1pen_problem} only guarantees feasibility of the optimal solution at each step. However, in this scenario we might be interested in the feasibility of the iterates $x_t$ generated by iteration of the form \eqref{eq:grad_proj_prox}. From \eqref{lem:track}, we know that in the limit the iterates will remain in some neighbourhood of the optimal points of each problem. In the next corollary this fact is exploited to define a new problem with tightened constraints (a smaller feasible set) that ensures feasibility of the iterates with high probability.
\begin{corollary}
	Let the sequence $x_t$ be generated by $$x_{t+1} = \prox{h_\theta}{\alpha}{x_t - \alpha \eta_t(x_t)},$$ where $h_\theta(x)= \lambda \sum_{i=1}^m \max (- c_i(x) + \theta (rv + \alpha \Delta)/({1-r}), 0 )$ for some positive scalar $\theta$ and $\lambda$ be selected such that the solution to
	$\min_{x\in \mathbb{R}^{n_x}} F_t(x) + h_\theta(x)$ is the same as the solution to $\min_{x\in \mathcal{X}_{\theta}} F_t(x)$ with $\mathcal{X}_{\theta} = \left \{x|c_i(x) \geq \theta({rv + \alpha \Delta})/({1-r}), i\in\mathcal{I} \right \}$.
	Then there exists a positive integer $\tau$ such that $\mathbb{P}(x_t \in\mathcal{X})$ for all $t\geq \tau$ is at least $1-1/\theta$.
\end{corollary}

%

\begin{figure}
	\centering
	\begin{tikzpicture}
	\node[] at (0,0) {\includegraphics[width=0.6\linewidth]{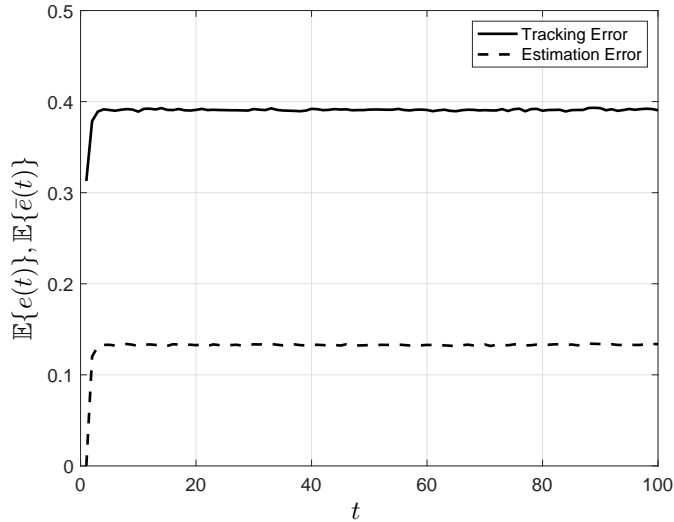}};
	\node[] at (0,-3.5) {$t$};
	\node[rotate=90] at (-4.4,0) {$\mathbb{E}\{e(t)\},\mathbb{E}\{\bar{e}(t)\}$};
	\end{tikzpicture}
	\caption{ \label{fig:simulation}
		Expected tracking Error $\mathbb{E}\{e(t)\}$ and estimation error $\mathbb{E}\{\bar{e}(t)\}$ versus time $t$.
		Expectations are approximated with 10,000 Monte Carlo simulations.
	}
\end{figure}

\section{Illustrative Example}
In this section, we consider an illustrative example with the aid of synthetic truncated Gaussian data. We say that random variable $x$ is distributed according to the truncated Gaussian random $\mathcal{N}_{\mathrm{trunc}}(\mu,\sigma,a_{\min},a_{\max})$ if its probability density function is given by
\begin{align*}
\left[\int_{a_{\min}}^{a_{\max}} \exp\left(-\frac{(x-\mu)^2}{2\sigma^2}\right) \mathrm{d}x \right]^{-1}\exp\left(-\frac{(x-\mu)^2}{2\sigma^2} \right)\mathds{1}_{a_{\min}\leq x\leq a_{\max}}.
\end{align*}
Consider a linear regression machine learning problem with input $w_1\in\mathbb{R}^{2}$ and output $w_2\in\mathbb{R}$. 
At time $t$, each entry of $w_1\in\mathbb{R}^2$ is a truncated Gaussian $\mathcal{N}_{\mathrm{trunc}}(0,1,-1,1)$ and $w_2$ is a truncated Gaussian $\mathcal{N}_{\mathrm{trunc}}(c_t w_1,1,-1,1)$ with
$c_t=
\begin{bmatrix}
\cos(\pi t/10) &
\sin(\pi t/10)
\end{bmatrix}.$ We use the fitness function $f(x,w)=\lambda \|x\|_2^2 +\|w_2-x^\top w_1\|_2^2$ for regression with $\lambda=10^{-2}$.  Now, we can check all items in Assumption~\ref{assum:assumption}.
Clearly, $f(x,w)$ is twice continuously differentiable with respect to $x$.
Also, $f(x,w)$, $\partial f(x,w)/\partial x_i$, $\partial^2 f(x,w)/\partial x_i\partial x_j$ are integrable with respect to $w$ for all $i,j$ and $\mathbb{E}^{w\sim \mathbb{P}_t}\{|f(x,w)|\}$, $\mathbb{E}^{w\sim \mathbb{P}_t}\{|\partial f(x,w)/\partial x_i|\}$, and $\mathbb{E}^{w\sim
	\mathbb{P}_t}\{|\partial^2 f(x,w)/\partial x_i\partial x_j|\}$ are bounded for all $i$, $j$, and $t$ due to continuity of derivatives and bounded support of the random variables. Note that $\nabla_x f(x,w)=2\lambda x -2(w_2-x^\top w_1)w_1$. We have $\|\nabla_x f(x,w)-\nabla_x f(x',w)\|\leq L_x\|x-x'\|$ for all $w$ with $L_x=2\lambda +  4$. Furthermore, if $\mathcal{X}=[-x_{\max},x_{\max}]^2$, we get $L_w=2(\lambda+2) x_{\max}+4$.
Also, $\nabla_x^2 f(x,w)=2\lambda I+w_1w_1^\top\succeq 2\lambda I$ with $\lambda>0$. We also meet Assumption~\ref{assum:subgaussian} for any $a>1$ because $w$ has bounded support. The final condition that we need to check is that  $\mathfrak{W}(\mathbb{P}_{t+1},\mathbb{P}_t)\leq \rho$ for all $t\in\mathbb{N}$. We find the value of $\rho$ for this condition numerically.
For each $t$, we generate $J\gg 1$ samples $\{w^{(j,t)}\}_{j=1}^J$ from $\mathbb{P}^t$ and then approximate $\rho$ by computing
\begin{align*}
\rho \approx \max_{1\leq t\leq  20}\mathfrak{W}\left(\frac{1}{J}\sum_{j=1}^J \delta_{w^{(j,t)}},\frac{1}{J}\sum_{j=1}^J \delta_{w^{(j,t+1)}}\right).
\end{align*}
Note that $c_t$ are periodic with period of $20$, i.e., $c_{t+20}=c_t$ for all $t$. This follows from that the first and the third term in the following inequality can be made arbitrary small for large $J$ with high probability~\cite{fournier2015rate}:
\begin{align*}
\mathfrak{W}(\mathbb{P}_{t},\mathbb{P}_{t+1}) \leq & \mathfrak{W}\left(\frac{1}{J}\sum_{j=1}^J \delta_{w^{(j,t)}},\mathbb{P}_{t}\right)+ \mathfrak{W}\left(\frac{1}{J}\sum_{j=1}^J \delta_{w^{(j,t)}},\frac{1}{J}\sum_{j=1}^J \delta_{w^{(j,t+1)}}\right)\\
&+\mathfrak{W}\left(\mathbb{P}_{t+1},\frac{1}{J}\sum_{j=1}^J \delta_{w^{(j,t+1)}}\right).
\end{align*}
Following this procedure, we get $\rho\approx0.33$ with $J=500$. In our simulations, we select $x_{\max}=2$. Let us select the step size $\alpha=0.4\in(0,0.4975)=(0,2/L_x)$. Figure~\ref{fig:simulation} shows expected tracking Error $\mathbb{E}\{e(t)\}$ and estimation error $\mathbb{E}\{\bar{e}(t)\}$ versus time $t$. As expected, the errors remain bounded; see Lemma~\ref{lem:track}.

\section{Conclusions and Future Work}
In this paper, we considered online stochastic optimization with drifting distributions. We  presented conditions on probability distributions, using the Wasserstein distance from the framework of data-driven distributionally-robust optimization, to model time-varying environments. We considered online proximal-gradient method to track the minimizers of expectations of smooth convex functions. We provided bounds on errors and regrets under strong convexity, Lipschitzness of the gradient, and bounds on the probability distribution drift. Noting that computing projections for a general feasible sets might not be amenable to online implementation (due to computational constraints), we proposed an exact penalty method. Following this, we relaxed the uniform boundedness of the gradient.  Future work can focus on online stochastic optimization over networks.

%

\bibliographystyle{ieeetr}
\bibliography{citation}

\appendix

\section{Proof of Lemma~\ref{lemma:Lipschitzness}}
\label{proof:lemma:Lipschitzness}
Note that
\begin{align*}
\|\nabla_x F_t(x)-\nabla_x F_t(x')\|
&=\|\nabla_x\mathbb{E}^{w\sim \mathbb{P}_t}\{f(x,w)\}-\nabla_x \mathbb{E}^{w\sim \mathbb{P}_t}\{f(x',w)\}\|\\
&=\|\mathbb{E}^{w\sim \mathbb{P}_t}\{\nabla_x f(x,w)\}- \mathbb{E}^{w\sim \mathbb{P}_t}\{\nabla_x f(x',w)\}\|,
\end{align*}
where the second inequality follows from Assumption~\ref{assum:assumption} (a)--(c). Furthermore,
\begin{align*}
\|\mathbb{E}^{w\sim \mathbb{P}_t}\{\nabla_x f(x,w)\}- \mathbb{E}^{w\sim \mathbb{P}_t}\{\nabla_x f(x',w)\}\|
&=\|\mathbb{E}^{w\sim \mathbb{P}_t}\{\nabla_x f(x,w)-\nabla_x f(x',w)\}\|\\
&\leq \mathbb{E}^{w\sim \mathbb{P}_t}\{\|\nabla_x f(x,w)-\nabla_x f(x',w)\|\}\\
&\leq L_x\|x-x'\|,
\end{align*}
where the first inequality follows from the Jensen's inequality and the second inequality is a consequence of Assumption~\ref{assum:assumption}~(e).

\section{Proof of Lemma~\ref{lem:opt_drift}}
\label{proof:lem:opt_drift}
Again, based on Assumption~\ref{assum:assumption} (a)--(c), we get
\begin{align*}
\|\nabla_x F_{t+1}(x)-\nabla_x F_t(x)\|
&=\|\nabla_x\mathbb{E}^{w\sim \mathbb{P}_{t+1}}\{f(x,w)\}-\nabla_x \mathbb{E}^{w\sim \mathbb{P}_t}\{f(x,w)\}\|\\
&=\|\mathbb{E}^{w\sim \mathbb{P}_{t+1}}\{\nabla_x f(x,w)\}- \mathbb{E}^{w\sim \mathbb{P}_t}\{\nabla_x f(x,w)\}\|.
\end{align*}
Furthermore,
\begin{align*}
\|\mathbb{E}^{w\sim \mathbb{P}_{t+1}}\{\nabla_x f(x,w)\}- \mathbb{E}^{w\sim \mathbb{P}_t}\{\nabla_x f(x,w)\}\|
&=\sum_{i=1}^{n_x} |\mathbb{E}^{w\sim \mathbb{P}_{t+1}}\{\partial f(x,w)/\partial x_i\}- \mathbb{E}^{w\sim \mathbb{P}_t}\{\partial  f(x,w)/\partial x_i\}|.
\end{align*}
Following~\eqref{eqn:dual} with Assumption~\ref{assum:assumption}~(f), we get
\begin{align*}
|\mathbb{E}^{w\sim \mathbb{P}_{t+1}}\{\partial f(x,w)/\partial x_i\}- \mathbb{E}^{w\sim \mathbb{P}_t}\{\partial  f(x,w)/\partial x_i\}|
\leq \mathfrak{W}(\mathbb{P}_t,\mathbb{P}_{t+1}) L_w.
\end{align*}
Finally, because of Assumption~\ref{assum:assumption}~(d), $\|\nabla_x F_{t+1}(x)-\nabla_x F_t(x)\|\leq \rho n_x L_w$. Assumptions~\ref{assum:assumption}~(a)--(c) and~(g) implies that $\nabla_x^2 F_t(x)\succeq \sigma I$. The rest of the proof follows from Lemma~4.2 in~\cite{selvaratnam2018numerical}.

\section{Proof of Lemma~\ref{lem:time_wise_est_bound}}
\label{proof:lem:time_wise_est_bound}
The proof follows the same steps as that of \cite[Proposition 6.1.8]{bertsekas2015convex} and the facts that the proximal operator is nonexpansive and $ x^*_t = \prox{h}{\alpha_t}{x^*_t - \alpha_t \nabla F_t(x^*_t)}$. Specifically, noting $\eta_t(x_t)= \nabla F_t (x_t) + \epsilon_t$, after straightforward algebraic manipulations and using Assumptions~\ref{assum:assumption} (e) and (g) one obtains
\begin{align*}
\|x_{t+1} - x^*_t\|^2   &\leq \|x_t - x^*_t\|^2 - 2\alpha_t (\nabla F_t(x_t) - \nabla F_t(x^*_t))^T (x_t - x^*_t) + \alpha_t^2 \|\nabla F_t(x_t) - \nabla F_t(x^*_t)\|^2 \\ &\quad\quad +  \alpha_t^2 \|\epsilon_t\|^2 + 2 \alpha_t \epsilon_t^T \left [ (x_t - \alpha_t \nabla F_t(x_t)) - (x^*_t - \alpha_t \nabla F_t(x^*_t)) \right ]\\
& \leq r_t^2 \|x_t - x^*_t\|^2 + 2 \alpha_t r_t \|\epsilon_t\| \|x_t - x^*_t\| + \alpha_t^2  \|\epsilon_t\|^2\\
&\leq (r_t \|x_t - x^*_t\| + \alpha_t \|\epsilon_t\|)^2.
\end{align*}



\section{Proof of Theorem~\ref{thm:regret}}
\label{proof:thm:regret}

Before proving Theorem~\ref{thm:regret}, as the sum of the errors for finite horizons are crucial to to bound regret, in the following two auxiliary lemmas, we present bounds on finite sums of errors.

\begin{lemma}\label{lem:sum_track}
	Let $\alpha$ be a constant scalar in $(0,2/L_x)$. Then
	\begin{align}
	\sum_{t=1}^T e_t &\leq \dfrac{1}{1-r}  \left [ (e_1 -r e_T -\alpha\|\epsilon_T\|) + \alpha E_T +V_T\right ],
	\label{eq:accum_track_err}
	\end{align}
	where $V_T=\sum_{t=1}^{T-1} \|x^*_{t+1} - x^*_t\|$, $E_T=\sum_{t=1}^T \|\epsilon_t\|$, and $e_t=\|x^*_t-x_t\|$.
\end{lemma}

\begin{proof} Summing both sides of \eqref{eq:time_wise_bound} yields
	\begin{align*}
	\sum_{t=1}^T e_t &\leq e_1 + \sum_{t=2}^T \|x_t-x^*_{t-1}\| + \sum_{t=2}^T \|x^*_t - x^*_{t-1}\| \\
	&\leq e_1 + r \sum_{t=2}^T e_{t-1}+ \alpha \sum_{t=2}^T \|\epsilon_{t-1}\| + \sum_{t=2}^T \|x^*_t - x^*_{t-1}\|\\
	&\leq (e_1 - r e_T) + r \sum_{t=1}^T e_t+ \alpha \sum_{t=2}^{T} \|\epsilon_{t-1}\| + \sum_{t=2}^T \|x^*_t - x^*_{t-1}\|\\
	\sum_{t=1}^T e_t &\leq \dfrac{1}{1-r}  \left [ (e_1 -r e_T) + \alpha \sum_{t=1}^{T-1} \|\epsilon_t\| +\sum_{t=1}^{T-1} \|x^*_{t+1} - x^*_t\|\right ].
	\end{align*}
\end{proof}

\begin{lemma}\label{lem:sum_est}
	Let $\alpha$ be a constant scalar in $(0,2/L_x)$. Then
	\begin{align}
	\sum_{t=1}^T \bar{e}_t &\leq \dfrac{1}{1-r}  \left [ \left ( \bar{e}_1 - r \bar{e}_T - \alpha \|\epsilon_1\| \right ) + \alpha E_T +r V_T \right ],
	\label{eq:accum_est_err}
	\end{align}
	where $V_T=\sum_{t=1}^{T-1} \|x^*_{t+1} - x^*_t\|$, $E_T=\sum_{t=1}^T \|\epsilon_t\|$, and $\bar{e}_t=\|x_{t+1}-x^*_t\|$.
\end{lemma}

\begin{proof}
	From Lemma \ref{lem:time_wise_est_bound}, we obtain
	\begin{align}\label{eq:time_wise_est_bound}
	\|x_{t+1}-x^*_t\| \leq r \|x_t-x^*_{t-1}\| + r\|x^*_t - x^*_{t-1}\| + \alpha \|\epsilon_t\|.
	\end{align}
	Summing both sides of \eqref{eq:time_wise_est_bound} results in
	\begin{align*}
	\sum_{t=2}^T \bar{e}_t &\leq r \sum_{t=2}^T \bar{e}_{t-1} + r \sum_{t=2}^T \|x^*_t - x^*_{t-1}\| +  \alpha \sum_{t=2}^T \|\epsilon_t\|
	\end{align*}
	Adding $\bar{e}_1$ to both sides, and adding and subtracting $r\bar{e}_T$ and $\alpha \|\epsilon_1\|$ to and from the right-hand-side yields
	\begin{align*}
	\sum_{t=1}^T \bar{e}_t &\leq \left ( \bar{e}_1 - r \bar{e}_T - \alpha \|\epsilon_1\| \right )+r \sum_{t=1}^T \bar{e}_{t} + r \sum_{t=1}^{T-1} \|x^*_{t+1} - x^*_t\| +  \alpha \sum_{t=1}^T \|\epsilon_t\|\\
	&\leq \dfrac{1}{1 - r} \left [ \left ( \bar{e}_1 - r \bar{e}_T - \alpha \|\epsilon_1\| \right ) + r \sum_{t=1}^{T-1} \|x^*_{t+1} - x^*_t\| +  \alpha \sum_{t=1}^T \|\epsilon_t\|  \right]
	\end{align*}
\end{proof}

Now we are ready to prove Theorem~\ref{thm:regret}. The necessary and sufficient condition for convexity of differentiable functions and applying the Cauchy-Schwarz inequality along with Assumption~\ref{assum:bounded_grad} leads to
\begin{align*}
F_t(x_t) - F_t(x^*_t)  &\leq  \nabla F_t(x_t)^T (x_t - x^*_t)\\
& \leq G \|x_t - x^*_t\|.
\end{align*}
Applying Lemma~\ref{lem:sum_track} for all $t\in\{1,\dots,T\}$ yields
\begin{align*}
\sum_{t=1}^T F_t(x_t) - F_t(x^*_t)  & \leq \dfrac{G}{1-r}  \left [ (e_1 -r e_T ) + \alpha \sum_{t=1}^{T-1} \|\epsilon_t\| +\sum_{t=1}^{T-1} \|x^*_{t+1} - x^*_t\|\right ].
\end{align*}
The bounds for ${\textup{Reg}}_T^{\textup{Estimation}}$ is proven similarly using Lemma~\ref{lem:sum_est}.

\section{Proof of Lemma~\ref{lem:dont_know_what_to_call_it}}
\label{proof:lem:dont_know_what_to_call_it}
From the first order optimality condition there exists a $g_{t+1} \in \partial h(x_{t+1})$ such that
\begin{align}
x_t - \alpha \nabla F_t(x_t) - \alpha \epsilon_t - x_{t+1} -  \alpha g_{t+1} = 0.
\end{align}
Consequently,
\begin{align*}
\|\nabla F_t(x_t) + \epsilon_t +   g_{t+1}\| & \leq \dfrac{1}{\alpha} \|x_{t+1} - x_t\|\\
&\leq  \dfrac{1}{\alpha} \left [ \|x_{t+1} - x^*_t\| +  \|x_t - x^*_{t-1}\| +  \|x^*_t - x^*_{t-1}\| \right ]\\
&\leq  \dfrac{1+r}{\alpha} \left [   \|x_t - x^*_{t-1}\| +  \|x^*_t - x^*_{t-1}\| \right ] + \|\epsilon_t\|\\
\mathbb{E} [\|\nabla F_t(x_t) + \epsilon_t +   g_{t+1}\|]&\leq \dfrac{1+r}{\alpha} \left [   r^{t-2} \mathbb{E}[\bar{e}_1 ]+ \dfrac{1 - r^{t-2}}{1-r} (\alpha \Delta + rv) +  v \right ] + \Delta\\
&\leq \dfrac{1+r}{\alpha} \left [ \mathbb{E}[\bar{e}_1] + \dfrac{\alpha \Delta + v}{1-r} \right ] + \Delta\\
&\leq \dfrac{1+r}{\alpha} \left [ r e_1 + \alpha \Delta + \dfrac{\alpha \Delta + v}{1-r} \right ] + \Delta
\end{align*}
Let $\Phi_t(x) = F_t(x) + h(x)$. Thus,
$\partial \Phi_t(x_t) = \nabla F_t(x) + \partial h(x_t)$.
For any $\varphi_t \in\partial\Phi_t(x_t)$ there exists a $\gamma_t\in \partial h(x_t)$ such that
$\varphi_t = \nabla F_t(x) + \gamma_t$.
Hence,
\begin{align*}
\mathbb{E} [\|\varphi_t\|] & =\mathbb{E} [ \| \nabla F_t(x_t) + \epsilon_t +   g_{t+1} + \gamma_t - \epsilon_t -   g_{t+1}\|]\\
&\leq \mathbb{E} [ \| \nabla F_t(x_t) + \epsilon_t +   g_{t+1} \| + \|\gamma_t\| +\| \epsilon_t\| + \|  g_{t+1}\|]\\
&\leq \overline{G},
\end{align*}
where
\begin{align*}
\overline{G} = \dfrac{1+r}{\alpha} \left [ r e_1 + \alpha \Delta + \dfrac{\alpha \Delta + v}{1-r} \right ] + 2\Delta +2\lambda n_c L_c.
\end{align*}

\end{document}